\documentclass{amsart}

\theoremstyle{plain}
\newtheorem{Prop}{Proposition}[section]
\newtheorem{Thm}[Prop]{Theorem}

\newtheorem{Lem}[Prop]{Lemma}

\theoremstyle{definition}
\newtheorem{Def}[Prop]{Definition}

\theoremstyle{remark}

\def\int{\mathop{\roman{int}}}

\def\1{^{-1}}

\def\asdim{\text{asdim}}

\def\T2{{\mathbf T_2}}

\def\UU{{\mathcal U}}

\def\VV{{\mathcal V}}

\errorcontextlines=0 \numberwithin{equation}{section}

\begin{document}
\title[On asymptotic dimension and a property of Nagata]{On asymptotic dimension and a property of Nagata}

\author{J.~Higes}
\address{Departamento de Geometr\'{\i}a y Topolog\'{\i}a,
Facultad de CC.Matem\'aticas. Universidad Complutense de Madrid.
Madrid, 28040 Spain} \email{josemhiges@yahoo.es}

\author{A.~Mitra}
\address{University of South Florida, St. Petersburg, FL 33701, USA}
\email{atish@stpt.usf.edu}

\date{ December 7, 2008}


\thanks{The first named author is supported by grant 2004-2494 from Ministerio Educacion y Ciencia, Spain and project MEC, MTM2006-0825. He thanks University of Tennessee for hospitality. Special thanks to J. Dydak and N. Brodskyi for helpful suggestions.}

\begin{abstract}
In this note we prove that every  metric space $(X, d)$ of asymptotic dimmension at most $n$ is coarsely equivalent to a metric space $(Y, D)$ that satisfies the following property of Nagata:\par
{\it For every $x,y_1,\cdots, y_{n+2} \in Y$ there exist $i,j\in \{1, \cdots, n+2\}$ with $i \ne  j$, such that $D(y_i,y_j)\le D(x,y_i)$.}\par
This solves problem $1400$ of \cite{Lviv}.
\end{abstract}

\maketitle

\tableofcontents
\section{Introduction}
Nagata introduced two properties to characterize metric spaces with finite topological dimension (see  \cite{Nagata}, \cite{Nagata2} and \cite{Nagata3}). Such properties are generalizations to higher dimensions of the notion of ultrametric space. The definition of the properties are the following ones:

\begin{Def} A metric space $(X, d)$ is said to satisfy the property $(N1)_n$  if for every $r >0$ and every $x, y_1,\cdots, y_{n+2}\in X$ such that $d(y_i, B(x, r))< 2\cdot r$ then there exists $i, j \in \{1, \cdots, n+2\}$ with $i \ne j$ such that $d(y_i, y_j) < 2\cdot r$.  
\end{Def} 

\begin{Def}
A metric space $(X, d)$ is said to satisfy property $(N2)_n$  if for every $x,y_1,\cdots, y_{n+2} \in X$ there exists $i,j\in \{1, \cdots, n+2\}$ with $i \ne  j$, such that $d(y_i,y_j)\le d(x,y_i)$.
\end{Def}

In \cite{Dran-Zar} Dranishnikov and Zarichnyi showed that every proper metric space $(X, d)$ of asymptotic dimension at most $n$ is coarsely equivalent to a proper metric space that satisfies $(N1)_n$. So a natural question is if the same statement is true  for the second property. Such problem appeared in \cite{Lviv} as problem $1400$. In this paper we solve it for general metric spaces using a technique of \cite{Brod-Dydak-Higes-Mitra}.

\section{Main theorem}
The notion of asymptotic dimension was introduced by Gromov in \cite{Gro asym invar} and it has been the focus of intense research in recent years. To give a definition we need to recall that a family of subsets $\UU$ of a metric space $(X, d)$ is said to be {\it $r$-disjoint} if for every two diferent $U \in \UU$ and $V \in \UU$ then $d(U, V) >r$.  
\begin{Def} A metric space $(X, d)$ is said to be of asymptotic dimension at most $n$ ($asdim (X, d) \le n$) if for every $r> 0$ there exists a uniformly bounded covering $\UU$ of subsets of $X$ such that $\UU$ splits in a union of the form $\UU = \bigcup_{i =1}^{n+1} \UU^i$ where each $\UU^i$ is an $r$-disjoint family of subsets. 
\end{Def} 

To simplify we will say that a covering $\UU$ of a metric space $(X, d)$  has {\it Lebesgue number} at least  $r$ ($L(\UU) \ge r$) if for every $x\in X$ the ball $B(x, r)$ is contained in one set of the covering. Next lemma characterizes asymptotic dimension in a nice way.  

\begin{Lem} \label{CoveringWithLebesgue} A metric space $(X, d)$  is of asymptotic dimension at most $n$ if and only if for every $r >0$ there exists a uniformly bounded covering $\UU$ where $\UU = \bigcup_{i= 1}^{n+1}\UU^i$ such that each $\UU^i$ is an $r$-disjoint family and  the Lebesgue number of $\UU$ is at least $r$.
\end{Lem} 
\begin{proof}
Only one implication is not trivial. Suppose $asdim(X, d) \le n$ and let $r$ be a positive number. As $asdim(X, d) \le n$ there exists a uniformly bounded covering $\VV' =\bigcup_{i=1}^{n+1}\VV^i$ such that each $\VV^i$ is a $3\cdot r$-disjoint family. Let $\UU^i$ be the family of subsets given by  $\UU^i = \{N(V, r)| V \in \VV^i\}$ where $N(V, r) = \{x | d(x, V) < r\}$. Now if $U, U' \in \UU^i$ are different elements of $\UU^i$ then clearly $d(U, U') > r$. Therefore the family $\UU = \bigcup_{i=1}^{n+1} \UU^i$ is a covering of mesh at most $mesh(\VV) + 2\cdot r$ so it satisfies the conditions of the lemma.   
\end{proof}

Next proposition is a version of proposition 3.6. of \cite{Dran-Zar}. We give a proof to make the paper self-contained.
\begin{Prop}\label{SequenceOfCoverings}
Let $(X, d)$ be a metric space  such that $asdim (X, d) \le n$. Then there exist a sequence of uniformly bounded coverings $\{\UU_k\}_{k =1}^{\infty}$ and an increasing sequences of numbers $\{d_k\}_{k =1}^{\infty}$ such that:
\begin{enumerate}
\item For every $k$  $\UU_k = \bigcup_{i = 1}^{n+1} \UU_k^i$ and each $\UU_k^i$ is a $d_k$-disjoint family. 
\item $L(\UU_k) \ge d_k$.
\item $d_{k+1} > 2 \cdot m_k$ where $m_k = mesh (\UU_k)$. 
\item  For every $i, k, l$ with $k <l$ and every $U \in \UU^i_k , V \in \UU^i_l$ , if $U \cap V \ne \emptyset$ , then $U \subset V$.
\end{enumerate}
\end{Prop}

\begin{proof}
Let us construct the sequence by induction on $k$. For $k = 1$ the result is just an application of lemma \ref{CoveringWithLebesgue} for $d_1 = 1$. Suppose we have a finite sequence of uniformly bounded coverings $\{\UU_k\}_{k=1}^t$ and a finite sequence of numbers $\{d_k\}_{k = 1}^t$ that satisfy properties (1)-(4). Let $d_{t+1} >0$ be a positive number such that $d_{t+1} > 2\cdot m_{t}$ and define $D_{t+1} = d_{t+1} + 2 \cdot m_{t}$. Hence by Lemma \ref{CoveringWithLebesgue} there exists an uniformly bounded covering $\VV_{t+1} = \bigcup_{i=1}^{n+1} \VV_{t+1}^i$ such that each $\VV_{t+1}^i$ is $D_{t+1}$-disjoint and $L(\VV_{t+1}) \ge D_{t+1}$. Now for every $i\in\{1,..., n+1\}$ we define the family of subsets $\UU_{t+1}^i = \{U_V| V \in \VV_{t+1}^i\}$ where $U_V$ is defined as the union of $V$ with all the subsets $W \in \UU_k^i$ with $k \in \{1, ..., t\}$ such that $W \cap V \ne \emptyset$. We claim that the covering given by $\UU_{t+1} = \bigcup_{i =1}^{n+1} \UU_{t+1}$ satisfies the required conditions. Clearly the mesh of $\UU_{t+1}$ is bounded by $2 \cdot m_t + mesh (\VV_{t+1})$. Also we have $L(\UU_{t+1})\ge L(\VV_{t+1})\ge D_{t+1} \ge d_{t+1}$. Let $i\in \{1,..., n+1\}$ be a fix number.  For every $U_V, U_{W}\in \UU_{t+1}^i$ we have $d(U_V, U_W) > d(V, W) - 2\cdot m_t = d_{t+1}$ so each $\UU_{t+1}^i$ is $d_{t+1}$-disjoint. The unique property that we need to check is (4). Let $k$, $l$ be two numbers such that  $k < l \le t+1$. If $l < t+1$ the fourth condition follows from the induction hypothesis. Let us suppose $l = t+1$. Let $W \in \UU_k^i$ and $U_V \in \UU_{t+1}^i$ such that $W \cap U_V \ne \emptyset$. This implies there exist an $s\le t$ and a $V' \in \UU_{s}^i$ such that $V' \cap V \ne \emptyset$ and 
$W \cap V' \ne \emptyset$. By the induction hypothesis  the last codition implies $W \subset V'$ or $V' \subset W$. In both cases we can conclude $W \subset U_V$. 

\end{proof}

 A map $f: (X, d_X) \to (Y, d_Y)$ between metric spaces is said to be a {\it coarse} map if for every $\delta >0$ there is an $\epsilon >0$ such that for every two points $a, b \in X$ that satisfy $d_X(a, b) \le \delta$ then $d_Y(f(a), f(b)) \le \epsilon$. If there exist also a coarse map $g: (Y, d_Y) \to (X, d_X)$ and a constant $C >0$ such that for every $x \in X$ and every $y \in Y$, $d_X(x, g(f(x))) \le C$ and $d_Y(y, f(g(y)))\le C$ then $f$ is said to be a {\it coarse equivalence} and the metric spaces $X$ and $Y$ are said to be {\it coarse equivalent}.

\begin{Thm} Every metric space $(X,d)$ with $\asdim X \le n$ is coarsely equivalent to a  metric space that satisfies property $(N2)_n$.
\end{Thm}

\begin{proof}
Let $\{\UU_k\}_{k=1}^{\infty}$ be a sequence of coverings as in proposition \ref{SequenceOfCoverings}. For every $x$ and $y$ in $X$ we define:
\[D(x, y) = \min\{k| \text{ there exists } U \in \UU_k \text{ such that } x, y \in U\}\]
Notice that by the second and third properties of $\UU_k$ we can easily deduce that for every $x, y , z \in X$: 
\[D(x, y) \le \max\{D(x, z), D(y, z)\}+1\]
This implies $(X, D)$ is a metric space.
Moreover if $D(x,y) \le \delta$ and $k$ is the minimum natural number such that $\delta \le k$  then $d(x, y) \le m_k$. In a similar way let us asume $d(x, y) \le \delta$. Let $d_r$ be the minimum number of the sequence$ \{d_k\}_{k =1}^{\infty}$ such that $\delta \le d_r$. We have $L(\UU_r) > d_r$, this implies there exists an $U\in \UU_r$ such that $x \in U$ and $y\in U$ what means $D(x, y) \le r$. We have shown $(X, D)$ is coarsely equivalent to $(X, d)$.\par
Now let $y_1,\cdots, y_{n+2}$, $x\in X$ and define for every $i \in\{1, \cdots,n+2\}$ the number $r(i) = D(y_i, x)$. By definition of $D$ we get that for every $i$ there exists a $V_i \in \UU_{r(i)}$ such that $y_i \in V_i$ and $x\in V_i$. For every $i$ there exists $k(i)$ such that $V_i \in \UU_{r(i)}^{k(i)}$. By the pigeon principle there are two $i, j$ with $i\ne j$ such that $k(i) = k(j)$. As $x \in V_i \cap V_j$ by the fourth property of proposition \ref{SequenceOfCoverings} we conclude $V_i \subseteq V_j$ or $V_j\subseteq V_i$. The first case means $D(y_i, y_j) \le D(x, y_j)$ and the second one means $D(y_j, y_i) \le D(x, y_i)$. Therefore $(X, D)$ satisfies $(N2)_n$.    
\end{proof}

\end{document}